\pdfoutput=1

\documentclass[a4paper,11pt]{amsart}
\usepackage{amsmath,amscd,amsfonts,amssymb,latexsym}
\usepackage{pgf}
\usepackage[utf8]{inputenc}
\usepackage{graphicx}
\usepackage{tikz}
\def\obrazek#1#2{\pgfdeclareimage[height=#2]{#1}{#1}
\begin{center}\pgfuseimage{#1}\end{center}}

\newtheorem{theorem}{Theorem}
\newtheorem{proposition}[theorem]{Proposition}
\newtheorem{lemma}[theorem]{Lemma}
\newtheorem{corollary}[theorem]{Corollary}
\newtheorem{remark}[theorem]{Remark}
\newtheorem{definition}[theorem]{Definition}
\newtheorem{example}[theorem]{Example}

\def\P{\mathbb{P}}
\def\C{\mathbb{C}}
\def\Q{\mathbb{Q}}
\def\R{\mathbb{R}}
\def\Z{\mathbb{Z}}
\def\qed{\hfill$\Box$}

\def\cat{{\rm cat}}
\def\sup{{\rm sup}}
\def\inf{{\rm inf}}
\def\eps{\varepsilon}
\def\GS{\mathcal{GS}}
\title[On the Pe{\l}czy\'{n}ski conjecture on Auerbach bases]
{On the Pe{\l}czy\'{n}ski conjecture on Auerbach bases}
\author{Andrzej Weber and Micha\l ~Wojciechowski}
\thanks{The first author is supported by NCN grant 2013/08/A/ST1/00804}
\address{Andrzej Weber, {\rm Department of Mathematics of Warsaw University\\Banacha 2, 02-097 Warszawa, Poland\\and Institute of Mathematics\\ Polish Academy of Sciences\\ ul. \'Sniadeckich 8, 00-656 Warszawa, Poland}}
\email{aweber@mimuw.edu.pl}
\address{Micha?? Wojciechowski, {\rm Institute of Mathematics\\ Polish Academy of Sciences\\ ul. \'Sniadeckich 8, 00-656 Warszawa, Poland}}
\email{miwoj@impan.pl}

\begin{document}
\begin{abstract}
We consider Auerbach bases in Banach spaces of dimension $n>2$. We show that there exist at least $(n-1)n/2+1$ such bases. This estimate follows from the calculation of the Lusternik--Schnirelmann category of the flag variety. A better estimate is obtained for generic smooth Banach spaces using Morse theory.
\end{abstract}
\maketitle
\section{Introduction}
Let $X$ be an $n$-dimensional Banach (real or complex) space. We call a biorthogonal system
of vectors $x_1,x_2,\dots,x_n\in X$ and functionals $x_1^*, x_2^*,\dots,x_n^*\in X^*$
an Auerbach basis provided $\|x_i\|=\|x_i^*\|=1$ for $i=1,2,\dots,n$. The existence
of Auerbach basis was established first by Auerbach in his PhD thesis
(cf. \cite{A}, \cite[Remarks to Chpt. VII]{Ban});
proof was first published independently by Day (cf. \cite{Day}) and Taylor (cf. \cite{Tay}).
It is based on the extremal argument - the basis vectors are selected by maximizing the
volume of the convex symmetric envelope of $n$-tuples of vectors from the unit sphere.
Plichko (cf. \cite{Pli}) noticed that if $X$ is not a Hilbert space the maximal argument applied to the dual space gives another Auerbach basis, which means that in a finite dimensional Banach space there always exist two different Auerbach bases  (here we identify bases which differ only by
permutation or multiplication by scalars of absolute value one). Motivated by this result  Pe{\l}czy\'{n}ski stated the conjecture that any $n$-dimensional Banach space admits at least $n$ different Auerbach bases. In this paper we confirm Pe{\l}czy\'{n}ski's conjecture by proving the following

\begin{theorem}\label{auerbach}
1) In every $n$-dimensional Banach space there exist at least $\frac12n(n-1)+1$ different
Auerbach bases

2) For an open (in the sense of Hausdorff distance) dense set of Banach spaces  with smooth (continually twice differentiable) norm, there
exist at least $2^{[n/2]}+4$ (real case) or $n!$ (complex case) different Auerbach bases.
\end{theorem}

The counterpart of Auerbach basis for non-symmetric convex bodies is the notion of conjugate affine diameters (cf. survey \cite{Soltan} for definitions). It was observed by Sobczyk (cf. \cite{Sobczyk}) that if $K$ is an $n$-dimensional convex body then any Auerbach basis of a Banach space with the unit ball $K-K$ corresponds to the system of conjugate affine diameters of $K$. Therefore we immediately get a corollary

\begin{corollary}
Any $n$-dimensional convex body has at last $\frac12n(n-1)+1$ different systems of conjugate affine diameters.
\end{corollary}

Our method of proof consists of the study of critical values of determinant function.
In Section \ref{section2} we define the function and establish a connection of its critical points with
Auerbach bases. In Section \ref{section3} we study  topological properties of the domain of the determinant function. In Section \ref{section4} we calculate the topological invariants which enable us to prove the second part of Theorem 1 via the Morse theory. In Section \ref{section5} we deal with the general case. We use the Lusternik--Schnirelmann theory to estimate the number of bifurcation points of the determinant function. Finally in Section \ref{section6} we show that the set of norms for which the Morse theory applies is open and dense in a suitable topology. Therefore  the estimation of the number of Auerbach bases given by the second part of Theorem \ref{auerbach} holds for a generic norm.  In Section 7 we apply the methods developed in the previous sections to estimate  the number of so called Auerbach simplices which could be treated as another analog of Auerbach bases for non-symmetric bodies.
\medskip

Our main estimate is based on the Lusternik--Schnirelmann theory which relates the category of a manifold with a number of critical points of a given function. In the  presentations which are available in the literature a function is assumed to be smooth. In the Appendix we give a proof with merely topological assumptions.
\medskip

We thank Zbigniew Szafraniec for suggesting application of Lusternik--Schnirelmann theory.

\section{Unit vector bases}\label{section2}
Let us fix a natural number $n>1$. Denote by $W(n)$ the set of square real matrices of size $n\times n$ the columns of which are unit vectors. Of course $W(n)$ is a product of spheres:
$$W(n)\simeq (S^{n-1})^n.$$
Let us consider the determinant of a square matrix as a function defined on $W(n)$. The value $0$ is a critical value.
Let us consider the open manifolds
$$W(n)_{>0}=\{B\in W(n)\,:\,\det(B)>0\}\,,$$
$$W(n)_{<0}=\{B\in W(n)\,:\,\det(B)<0\}\,.$$
These manifolds are homeomorphic, and so let us focus on $W(n)_{>0}$.

\begin{proposition}\label{Wzero} The manifold $W(n)_{>0}$ is diffeomorphic to $$SO(n)\times \mathring{\mathbb B}^{\frac{n(n-1)}2}\,,$$
where $\mathring{\mathbb B}^{\frac{n(n-1)}2}$ is an open ball of the dimension $\frac{n(n-1)}2$.
\end{proposition}
\begin{proof} The group $SO(n)$ acts on $W(n)_{>0}$ by multiplication from the left. Consider the map  $$\GS :W(n)_{>0} \to SO(n)$$ given by the Gram-Schmidt orthogonalization
$$\GS (v_1,v_2,\dots,v_n)= \left(v_1,\frac{v_2-\langle v_1,v_2\rangle v_1}{|v_2-\langle v_1,v_2\rangle v_1|},\dots\right)\,.$$
This map is well defined since $W(n)_{>0}$ consists of non-degenerate matrices. The map $\GS$ is invariant with respect to the left action of $SO(n)$. This means, that for $A\in SO(n)$ and $B=(v_1,v_2,\dots,v_n)\in W(n)_{>0}$ we have
$$\GS (A\cdot B)=A\cdot \GS (B)\,.$$
It follows that $$W(n)_{>0}\simeq SO(n)\times \GS ^{-1}(I)\,,$$
where $I$ is the identity matrix.
We claim that $\GS^{-1}(I)$ is homeomorphic to the open ball $\mathring{\mathbb B}^{\frac{n(n-1)}2}$. To see this
observe, that $\GS^{-1}(I)$  consists of matrices of the shape
$$\left(\begin{matrix}
1&a_{12}&a_{13}&\dots&a_{1n}\\
0&\sqrt{1-a_{12}^2}&a_{23}&\dots&a_{2n}\\
0&0&\sqrt{1-a_{13}^2-a_{23}^2}&\dots&a_{3n}\\
\vdots\\
0&0&0&\dots&\sqrt{1-\sum_{i=1}^{n-1}a_{in}^2}\end{matrix}\right)\,.$$
Hence $\GS^{-1}(I)$ can be identified with the subset of $\R^{\frac{n(n-1)}2}$ with coordinates $a_{ij}$ for $1\leq i<j\leq n$  on which the square roots is defined. That is we demand that
\begin{equation}1-\sum_{i=1}^{j-1}a_{ij}^2>0\label{nierownosci}\end{equation}
for $j=2,3,\dots,n$. This set  is the product of the unit open balls  $\mathring{\mathbb B}_j\subset\R^{j-1}$ (where $\R^{j-1}$ has coordinates $a_{ij}$, $1\leq i<j$).
 The product of open balls is diffeomorphic to an open ball.\end{proof}

Let $$G=\Z_2^n\rtimes \Sigma_n$$ be the semidirect product of the permutation group $\Sigma_n$ and the group $\Z_2^n$.
This group is generated by matrices of permutations and the diagonal matrices with $\pm1$ on the diagonal. Every element of $G$ can be uniquely written as $a\sigma$, where $a$ is a diagonal matrix and $\sigma$ is a permutation. The group $G$ acts on $W(n)$ by multiplication from  the right, when we view the elements of $W(n)$ as matrices or by  permutation and change of sign, when we view the elements of $W(n)$ as $n$-tuples of unit vectors. The group $G_0\subset G$ consisting of matrices of determinant equal to 1 preserves $W(n)_{>0}$. The group $G_0$ consists of elements $a\sigma\in G$ for which $\det(a)=\det(\sigma)$. The group $G_0$ acts freely on $W(n)_{>0}$ since $G_0$ and $W(n)_{>0}$ are subgroups of the linear group $GL_n(\R)$ and the action is given by multiplication of matrices. Therefore the quotient $W(n)_{>0}/G_0$ is a manifold.

\begin{example}\label{przyklad}\rm Let $n=2$. Then $W(n)_{>0}$ is homeomorphic to $SO(2)\times (0,\pi)$. The pair $(R_\phi,\theta)$ is sent to the pair of vectors $$\big(\cos(\phi),\sin(\phi))\,,\;(\cos(\phi+\theta),\sin(\phi+\theta)\big)\,.$$
Here $R_\phi$ is the rotation by the angle $\phi\in[0,2\pi)$.
The group $G_0$ is generated by the elements:
\begin{align*}\alpha:(R_\phi,\theta)&\mapsto (R_{\phi+\pi},\theta)\\
\beta:(R_\phi,\theta)&\mapsto (R_{\phi+\theta},\pi-\theta)\end{align*}
The quotient  $W(n)_{>0}/G_0$ is homeomorphic to the open M\"obius strip.
\end{example}

A finite dimensional Banach space is identified with fixed Euclidean space $\R^n$ together
with its closed unit ball - a convex centrally symmetric body $D\subset\R^n$ with nonempty interior. Given such  body
we will construct a  function $g:W(n)/G_0\to \R$ whose critical or strictly speaking bifurcation points will be correspond to Auerbach bases of the underlying Banach space.

\begin{definition}\label{bifur} We say that $x\in M$ is a topologically regular point of $g$ if there exists a neighbourhood of $x$ which is of a product form $U\simeq S\times (a-\eps,a+\eps)$ and the function $g$ on $U$ coincides with the projection onto the second factor. If $x$ is not topologically regular, then we say that it is a bifurcation point. The corresponding value $f(x)$ is called a bifurcation value.\end{definition}

If $g$ is a smooth function, then a bifurcation point is a critical point in the differential sense. There are critical points which are topologically regular (for example the function $g(x)=x^3$). By the Morse lemma the critical points with nondegenerate Hessian (i.e. Morse singularities) are bifurcation points.

\bigskip

 Let  $\partial D$ be the boundary of $D$. Of course $\partial D$ is homeomorphic to $S^{n-1}$.
For a vector $v\in S^{n-1}$ let $h(v)$ be the length of a vector which is proportional to $v$ and belongs to $\partial D$.
Let us define
\begin{equation}\label{gtilde}\tilde g=\tilde g_D:W(n)=(S^{n-1})^n\to \R\end{equation}
$$\tilde g(v_1,v_2,\dots,v_n):=\det(v_1,v_2,\dots,v_n)\prod_{i=1}^nh(v_i)\,.$$
This function is equal to the volume of the parallelepiped spanned by the vectors in $\partial D$. Let
$$g:W(n)_{>0}/G_0\to \R_+$$
be the resulting map from the quotient.

The next theorem connects the Auerbach bases to the bifurcation points of $g$.

\begin{theorem} Let $g$ be defined as above. Then if $\{w_1,w_2,\dots,w_n\}\in W_{>0}/G_0$ is a topological bifurcation point, then
$w_1,w_2,\dots,w_n$ has Auerbach property, i.e. the vectors
$$h(w_1)w_1, h(w_2)w_2,\dots,h(w_n)w_n$$
 form an Auerbach basis of $D$. Different bifurcation points correspond to different Auerbach bases.\end{theorem}
\begin{proof}
We have a convex body $D\subset\R^n$, which is centrally symmetric. Suppose that the vectors
$w_1,w_2,\dots,w_{n-1}$ do not form an Auerbach basis. Then, without loss of generality
we can assume that the hyperplane passing through $w_n$ parallel to $w_1, w_2,\dots,w_{n-1}$ does not support $D$ at $w_n$.
Let
\begin{multline*}U=\{(v_1,v_2,\dots,v_n)\in(\partial D)^n\;|\;v_1,v_2,\dots,v_{n-1}\text{ are linearly independent, and }\\ v_n+span\{(v_1,v_2,\dots,v_{n-1}\}\text{ is not a supporting hyperplane at }v_n\}\,.\end{multline*}
Let
$$V=\{(v_1,v_2,\dots,v_{n-1})\in(\partial D)^{n-1}\;|\;v_1,v_2,\dots,v_{n-1}\text{ are linearly independent}\}\,.$$
Let $Q\subset V\times \R^n$ be the sphere bundle  over $V$, whose fiber over
$(v_1,v_2,\dots,v_{n-1})\in V$ is the unit sphere in $span\{v_1,v_2,\dots,v_{n-1}\}$, i.e. we have a fibration
$$S^{n-2}\hookrightarrow Q\twoheadrightarrow V\,.$$
We define a map $$(f_1,f_2):U\to Q\times \R$$ as follows:\begin{itemize}
\item let $P=v_n+span\{v_1,v_2,\dots,v_{n-1}\}$ be the affine hyperplane in $\R^n$, which is parallel to $span\{v_1,v_2,\dots,v_{n-1}\}$ an passes through $v_n$; by the assumption it has nonempty intersection with $int D$,
\item let $b$ be the barycenter of $P\cap D$; clearly $b$ belongs to the relative interior of $(P\cap D)$,
\item let $S(b)\simeq S^{n-2}$ be the unit sphere in the hyperplane $P$ with center in $b$,
\item we  project radially $v_n$ on $S(b)$,
\item we obtain a unit vector
$$f(v_1,v_2,\dots,v_{n-1}):=\tfrac{v_n-b}{|v_n-b|}\in span\{v_1,v_2,\dots,v_{n-1}\}$$
and we put $f_1(v_1,\dots,v_{n-1})=(v_1,\dots,v_{n-1}, f).$
\end{itemize}
In this way we get a continuous section of the bundle $Q$.
The $\R$-coordinate $f_2$ is equal to
$$\det(v_1,v_2,\dots,v_{n})=\pm \,vol(v_1,v_2,\dots,v_{n-1})\cdot dist(v_n,span\{v_1,v
_2,\dots,v_{n-1}\})\,.
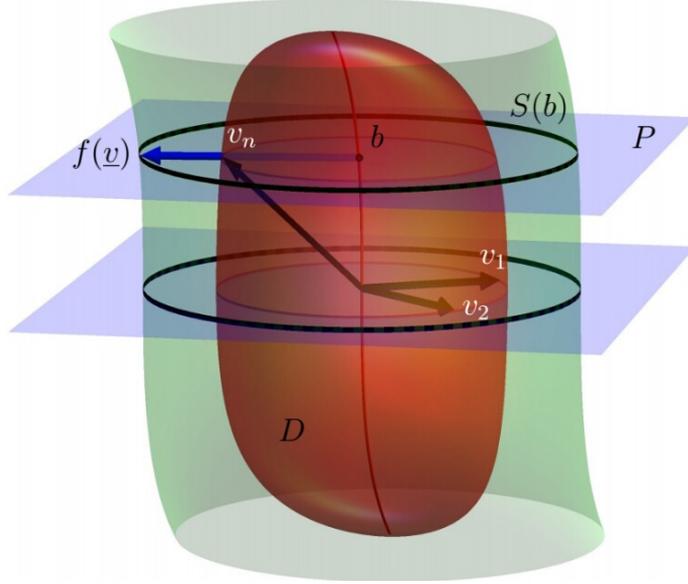
$$

\obrazek{wwwrys}{8cm}

We show that the resulting function is a homeomorphism onto the image. The image is open. Thus it is locally the product, as in the definition of a regular point of the mapping $f_2=\det:U\to\R$.

The inverse map is the following: suppose $$((v_1,v_2,\dots,v_{n-1}),p)\in Q\subset V\times \R^n\,,$$
i.e.
$$p\in span\{ v_1,v_2,\dots,v_{n-1}\}\,,\quad |p|=1.$$
and let $a\in \R$ be a real number such that
$$int(D\cap P)\not=\emptyset\,,$$
where $P$ is the affine hyperplane in $\R^n$ which is parallel to $span\{v_1,v_2,\dots,v_{n-1}\}$ and $$\det(v_1,v_2,\dots,v_{n-1},v)=a$$ for any $v\in P$.
(This exactly means that $(v_1,v_2,\dots,v_{n-1},p,a)$ is in the image of the map $(f_1,f_2)$.)
Let $b$ be the barycenter of $D\cap P$. Define $v_n$ as the central projection of $b+p$ form the center in $b$ to $\partial D$.

The last part of the statement follows from the fact that dividing by $G_0$ identifies bases which differ by permutation or multiplication by numbers of absolute value one.
\end{proof}

Suppose that $g$ is a differentiable function. Not every
critical point of $g$ is a bifurcation point, but it can be easily shown that:
\begin{proposition}
If $g$ is a twice differentiable Morse function then any critical point corresponds to an Auerbach basis.
\end{proposition}
We leave the argument to the reader, it follows from the topological description of the neighborhood of the critical point of the function (Morse lemma).

\section{Flag varieties}\label{section3}
Let us consider the quotient
$$W(n)_{\not=0}/\Z_2^n\simeq W(n)_{>0}/\Z_2^{n-1}\,.$$
The Gram-Schmidt orthogonalization $\GS:W(n)_{\not=0}\to O(n)$ commutes with the right action of $\Z_2^n$:
$$\GS(B\cdot A)=\GS(B)\cdot A$$
for $B\in W(n)_{\not=0}$ and a diagonal matrix $A\in \Z_2^n\subset O(n)$.
That is so because if
$$\GS(v_1,v_2,\dots,v_n)=(w_1,w_2,\dots,w_n)$$
then $w_k$ is the normalized projection of $v_k$ on $span(v_1,v_2,\dots,v_{k-1})^\perp$. If we change the sign of $v_k$, then none of $span(v_1,v_2,\dots,v_{\ell})$ is changed, $w_\ell$ remains the same except for $w_k$ which changes the sign.
Therefore we obtain the decomposition
$$ W(n)_{\not=0}/\Z_2^n\simeq O(n)/\Z_2^n\times
\mathring{\mathbb B}^{\frac{n(n-1)}2}$$
or
$$W(n)_{>0}/\Z_2^{n-1}\simeq SO(n)/\Z_2^{n-1}\times
\mathring{\mathbb B}^{\frac{n(n-1)}2}$$
The space $O(n)/\Z_2^n$ parameterizes  the collections of $n$ perpendicular lines in $\R^n$. This space can be identified with the real flag variety $Fl(n,\R)$.

\section{cohomology}\label{section4}
The cohomology  of the flag variety with $\Z_2$ coefficients is well known:

\begin{theorem}\label{flagiR} The cohomology $H^*(Fl(n,\R);\Z_2)$ is a $\Z_2$ algebra generated by elements from the first gradation
(Stiefel-Whitney classes of the tautological line bundles). The dimension is equal to $n!$.\end{theorem}

\begin{remark}\rm The statement of Theorem \ref{flagiR} follows inductively from the Leray-Hirsch Theorem \cite[Ch. 17.1]{Hus} applied to the fibration
$$Fl(n-1,\R)\hookrightarrow Fl(n,\R) \twoheadrightarrow \R\P^{n-1}\,.$$
In fact
$$H^*(Fl(n,\R);\Z_2)\simeq \Z_2[x_1,x_2,\dots,x_n]/(\sigma_1,\sigma_2,\dots,\sigma_n)\,,$$
where $x_i$ are generators in degree one corresponding to the Stiefel-Whitney classes of the tautological line bundles, and $\sigma_i=\sigma_i(x_1,x_2,\dots,x_n)$ are the elementary symmetric functions. It is hard to give a precise
reference to that fact. It follows from the corresponding statement for the complex flag variety, cf. \cite[20.3(b)]{Bor}.\end{remark}

We compute the cohomology with rational coefficients of $W(n)_{>0}/G_0$  applying  the following result:
\begin{proposition}\cite[III.7.2]{Br}\label{bredon} Let $X$ be a topological paracompact space with an action of a finite group $G$. The cohomology of the quotient is isomorphic to the invariant part of the original cohomology:
$$H^*(X/G;\Q)\simeq H^*(X;\Q)^{G}\,.$$\end{proposition}
We apply this general rule in our situation:
\begin{theorem}There is an isomorphism
$$H^*(W(n)_{>0}/G_0;\Q)\simeq H^*(SO(n);\Q)\,.$$
\end{theorem}
\begin{proof}By Proposition \ref{Wzero} the space $W(n)_{>0}$ retracts to $W(n)_1=SO(n)$. Observe that the action of $G_0$ on $H^*(SO(n))$ is trivial, since every element of $G_0$ can be connected with the identity by a path in $SO(n)$.
By Proposition \ref{bredon} we obtain
$$H^*(W(n)_{>0}/G_0,\Q)=H^*(W(n)_{>0};\Q)^{G_0}=H^*(SO(n);\Q)^{G_0}=H^*(SO(n);\Q)\,.$$
\end{proof}

Since the rational cohomology of $SO(n)$ is the same as the cohomology of the product of $\lfloor\frac n2\rfloor$ spheres (cf. \cite[Prop. 3D4]{Hat} ), we obtain:
\begin{corollary}

$$\dim(H^*(W(n)_{>0}/G_0;\Q))=2^{\lfloor\frac n2\rfloor}$$

\end{corollary}

\begin{corollary}\label{morseoszac} Let $$g:W(n)_{\geq 0}/G_0\to \R$$ be a   function such that $g^{-1}(0)=W(n)_{0}/G_0$ and $g$ is a Morse function on $W(n)_{>0}/G_0$. Then $g$ has at least $2^{\lfloor\frac n2\rfloor}$ critical points.
\end{corollary}

\begin{proof}Let $M=W_{>0}/G_0$.
Suppose $g$ has only a finite number of critical points (otherwise we are done). Then for $\eps$ sufficiently close to 0 the homotopy type of $M_{>\eps}$  is the same as $M$. We apply Morse theory for  the function $-g$. For sufficiently large $a\in \R$ we have $M_{\geq a}=\emptyset$ and between $a$ and $\eps$ there are critical values $t$, each of which provokes a change of $\dim(H^*(M_{>t}))$ by one. It follows that there have to be at least $\dim(H^*(M;\Q))$ critical points.\end{proof}

Using properties of the fundamental group we can improve a little bit the bound from the above  corollary.
The fundamental group  $$\pi_1(W(n)_{>0}/G_0)=\pi_1(SO(n)/G_0)$$ fits  into the short exact sequence
$$0\to\Z_2\to\pi_1(SO(n)/G_0)\to G_0\to 0\,.$$
For $n\geq 3$ this group is not abelian, since the quotient $G_0$ is not abelian. (The cyclic permutation of coordinates does not commute with diagonal matrices.)

\begin{corollary} Suppose $n\geq3$. Then $g$   satisfying the assumption of Corollary \ref{morseoszac} has at least $2^{\lfloor\frac n2\rfloor}+4$ critical points.
\end{corollary}

\begin{proof}A critical value $t$, provokes a modification of the homotopy type of $M_{>t}$. It follows that there has to be at least two critical points of index 1 because such points correspond to the generators of $\pi_1(M_{>\eps})$ (note that $\pi_1$ is not generated by a single element). Since  $H_1(SO(n),\Q)=0$, there has to be at least two critical points of index 2 which kill the generators of $\pi_1$ in $H_1$. The remaining critical points come from the generators of  homology of degree at least three.\end{proof}

We deal now with the complex case.
Here instead of $W$ we study a collection of $n$ unit vectors in $\C^n$. Denote this space by $W(n)^\C$ (it is the product of spheres $(S^{2n-1})^n$). The Gram-Schmidt orthogonalization can be considered as a map
$$\GS:W(n)^\C_{\not=0}\to U(n)\,.$$
where $W(n)^\C_{\not=0}$ is the set of linearly independent vectors, which is identified with the set of matrices with determinant not equal to zero. As before we have
\begin{proposition}
$$W(n)^\C_{\not=0}\simeq U(n)\times
\mathring{\mathbb B}^{n(n-1)}\,.$$
\end{proposition}

Here instead of $\Z_2^n$ the torus $(S^1)^n$ acts via rotating the vectors. (In the real case we could only change the sign.) As before this action commutes with the Gram-Schmidt orthogonalization process. We obtain

\begin{proposition}\label{opiszesp}
$$W(n)^\C_{\not=0}/(S^1)^n\simeq U(n)/(S^1)^n\times
\mathring{\mathbb B}^{n(n-1)}\,.$$
\end{proposition}

The space $U(n)/(S^1)^n$ is the complex flag variety $Fl(n,\C)$. The integral cohomology of the flag variety is free and additively generated by the Schubert cycles:

\begin{theorem}\cite[Th. 20.3(b)]{Bor} The cohomology $H^*(Fl(n,\C);\Z)$ is an  algebra generated by elements from the second gradation
(Chern classes of the tautological line bundles). The rank is equal to $n!$.\end{theorem}

Combining the result above, Proposition \ref{opiszesp} with Theorem \ref{bredon} we obtain

\begin{corollary} The cohomology $H^*(W(n)^\C_{\not=0}/((S^1)^n\rtimes\Sigma_n);\Q)$ is an  algebra generated by elements from the second gradation. The dimension (i.e. the sum of Betti numbers) is equal to $n!$.
\end{corollary}

\begin{corollary} Let $$g:W(n)^\C/((S^1)^n\rtimes\Sigma_n)\to \R$$ be a   function such that $g^{-1}(0)=W(n)^\C_{0}/((S^1)^n\rtimes\Sigma_n)$ and $g$ is a Morse function on $W(n)_{\not=0}/G_0$. Then $g$ has at least $n!+2$ critical points.
\end{corollary}

The proof is the same as in the real case except the argument involving the fundamental group. We know that there are at least two critical points of index one, but we do not control the number of points of index two.
\begin{remark}\rm We do not have to assume that $D$ is convex. It is enough to assume that it is star-shaped and centrally symmetric.\end{remark}

\section{Estimate based on Lusternik--Schnirelmann category}
\label{section5}

If a function $g$ is not of Morse type, then we can apply Lusternik--Schnirelmann category to estimate the number of bifurcation points (cf. \cite{LS} and also a modern monograph \cite{CLOT}). 

\begin{definition} Let $X$ be a topological space and $Y\subset X$ its closed subspace.
Then the Lusternik--Schnirelmann category $\cat_X(Y)$ denots the smallest cardinality of
covering of $Y$ by open sets which are contractible in $X$.
\end{definition}

We need the following version of the Lusternik--Schnirelman theorem which is valid for general continuous maps.

\begin{theorem}\label{LSmain} Let $M$ be a path connected metric space which is locally contractible. Let $f:M\to \R$ be a continuous proper function which is bounded from below. For each $a\in \R$ the number of bifurcation points with $f(x)\leq a$ is not smaller than $\cat_M(M_{\leq a})$.\end{theorem}

In the  presentations of the Lusternik--Schnirelmann theory which are available in the li\-terature (see also \cite{Fox, J, DFN}) $M$ is assumed to be a manifold and $f$ a smooth function. In the Appendix we give a proof of Theorem \ref{LSmain} as stated with only topological assumptions.

\medskip

 We will apply the above theorem for $M=W(n)_{>0}/G_0$ and the function
for $$f=-\log(g(v_1,v_2,\dots, v_n))_{|W(n)_{>0}}$$ (see the formula (\ref{gtilde})).
The  functions $f$ and $g$ have the same bifurcation points but $f$ is proper and bounded from below as assumed in Theorem \ref{LSmain}. The manifold
$M$ retracts to $SO(n)/G_0$, so $\cat(M)=\cat_M(SO(n)/G_0)$. The space $SO(n)/G_0$ is compact, thus it is contained in some $M_{\leq a}$. Hence there are at least $\cat(M)$ bifurcation points  with values $\leq a$.

We will use the cup-length to estimate the $\cat(W(n)_{>0}/\Z_2^n)$.

\begin{definition} Let $R$ be a ring. The cup-length
$$\ell_R(M)$$ is the length of the longest sequence of  $\alpha_1,\alpha_2,\dots,\alpha_\ell$ of cohomology classes $\alpha_j\in H^*(M,\R)$ of positive degree such that
$$\alpha_1\cup\alpha_2\cup\dots\cup\alpha_\ell\not=0\,.$$\end{definition}

Cup-length is a basic lower bound for the category of a set,
cf. \cite[(1.3)]{J} or \cite[Theorem 1.15]{CLOT} for related historical remarks and further references.

\begin{theorem} Let $M$ be a topological space. Then for any ring $R$
$$\cat(M)\geq \ell_R(M)+1\,.$$
\end{theorem}

The real flag variety is of dimension $\frac{n(n-1)}2$. Its fundamental class in $H^{\frac{n(n-1)}2}(Fl(n,\R))$ is a product of classes of gradation one by Theorem \ref{flagiR}. Thus $$\ell_{\Z_2}(Fl(n,\R))=\tfrac{n(n-1)}2\quad\text{and}\quad \cat(Fl(n,\R))\geq\tfrac{n(n-1)}2+1.$$
In fact $\cat(Fl(n,\R))$ cannot exceed its dimension, \cite[Prop. 2.1]{J}. Thus
$$\cat(Fl(n,\R))=\tfrac{n(n-1)}2+1.$$

So far we have computed the category of the space $W(n)_{>0}/\Z_2^n\stackrel{htp}\sim Fl(n,\R)$, but we are interested in a quotient of that manifold.
The category of the quotient space is always not smaller than the original one under some general topological assumptions (e.g. that the space is locally contractible).

\begin{theorem}\label{26}\cite[ Th. 21.1]{Fox} Let $M_1\to M_2$ be a covering map of topological spaces. Then $\cat(M_2)\geq \cat(M_1)$.\end{theorem}

Hence
$$\cat(W(n)_{>0}/G_0)= \cat(SO(n)/G_0)\geq \cat(Fl(n,\R))=\frac{n(n-1)}2+1.$$
By the dimension argument we have an equality.
\begin{corollary}
Any function satisfying the assumption of Theorem \ref{LSmain} has at least $\frac{n(n-1)}2+1$ bifurcation points.
\end{corollary}
By the same method we obtain an estimate for the complex case
\begin{corollary} We have $$\cat(W(n)_{>0}^\C/(S^1)^n\rtimes \Sigma_n)\geq\frac{n(n-1)}2+1.$$
Any function satisfying the assumption of Theorem \ref{LSmain} has at least $\frac{n(n-1)}2+1$ bifurcation points.
\end{corollary}

\section{Generic convex bodies}\label{section6}

The second part of Theorem 1 follows directly from the combination of the next three statements.

1) Morse functions are open in $C^2\left(W(n)_{>0}/G_0)\right )$.

2) Banach spaces with semialgebraic unit ball are dense in space parameterizing Banach spaces equipped with the Hausdorff metric. This is a result of Hammer (cf. \cite{Ham} or \cite{Kroo} for quantitative version).

3) In the algebraic variety of convex symmetric surfaces $D$ which are given by the level surfaces of polynomials of $n$ variable of fixed degree, those for which the determinant function $g$ is not Morse form a proper subvariety.

We will give a proof of the last statement, namely
we will show that for generic convex semialgebraic bodies the function $\tilde g$ is a Morse function, hence there exist  at least $\geq 2^{\lfloor\frac n2\rfloor}+4$ Auerbach bases. We will give an argument for the bodies defined by homogeneous polynomial equations.

\begin{example}\rm Let $D$ be the convex body given by the inequality $P_m(x)=\sum_{i=1}^nx_i^{2m}\leq1$. The standard basis $\eps_1,\eps_2,\dots,\eps_n$ satisfies the Auerbach condition. This means that $\eps_1,\eps_2,\dots,\eps_n$ is a critical point of the function $\det:(\partial D)^n\to \R$. We claim that this critical point is  nondegenerate. To show that we compute the Hessian. We choose coordinates in $(\partial D)^n$ indexed by pairs $(i,j)$, $i,j\in\{1,2,\dots,n\}$, $i\ne j$. We set
$$v_i=(x_{i,1},x_{i,2}, \dots ,x_{i,n}), \quad\text{with } x_{i,i}=(1-\sum_{j\ne i}x_{i,j}^{2m})^{1/{2m}}\,. $$
Then for $m>1$
\begin{multline*}\det(v_1,v_2,\dots,v_n)=\det\left(\left(\begin{matrix}
1&x_{1,2}&x_{1,3}&\dots&x_{1,n}\\
x_{2,1}&1&x_{2,3}&\dots&x_{2,n}\\
x_{3,1}&x_{3,2}&1&\dots&x_{3,n}\\
\vdots\\
x_{n,1}&x_{n,2}&x_{n,3}&\dots&1\end{matrix}\right)+{\mathcal  O}(||x||^{2m})\right)=\\=1-\sum_{i\ne j} x_{i,j}x_{j,i}+{\mathcal  O}(||x||^{3})\,.\end{multline*}
This symmetric form is a sum of standard hyperbolic forms, so it is nondegenerate.

Moreover  every  homogeneous polynomial $Q$ of degree $2m$ having a critical point in $(\eps_1,\eps_2,\dots,\eps_n)$ can be deformed to $P_m$  $$Q_t=(1-t)Q+t\sum x_i^{2m}\,.$$ For $t=1$ we have a Morse singularity and the family is analytic. Therefore sufficiently close to $t=0$ the singularity is Morse as well.\end{example}

Let $V_{n,m}$ be the space of  homogenous polynomials in $n$ variables of degree $2m$. Let $U_{n,m}$ be the subset of polynomials $P$, such that $D_P=\{v\in\R^n\;|\:P(v)\leq 1\}$ is strictly convex. The set $U_{n,m}$ is an open set in $V_{n,m}$.
For the convex body defined by $P(v)\leq 1$ the function $g_P^+=\tilde g_{D_P}:W(n)_{\leq 0}\to \R$ (see (\ref{gtilde})) is equal to
$$g_P^+=\tilde g_{D_P}(v_1,v_2,\dots,v_n)=\det(v_1,v_2,\dots,v_n)\big/\prod_{i=1}^n P(v_i)^{\frac1{2m}}\,.$$

We define the critical and degeneration sets:
$${\bf Crit}_{n,m}=\{(v_1,v_2,\dots,v_n,P)\in W(n)_{>0}\times U_{n,m}\;|\;(v_1,v_2,\dots,v_n)\text{ is a critical point of } g_P^+\;\}$$
\begin{multline*}{\bf Degen}_{n,m}=\{(v_1,v_2,\dots,v_n,P)\in W(n)_{>0}\times U_{n,m}\;|\;
(v_1,v_2,\dots,v_n)\text{ is a degenerate} \\\text{critical point of } g_P^+\;\}\,.\end{multline*}
 Defined above sets are given by analytic equations.
For every positively oriented $n$-tuple $\underline v=(v_1,v_2,\dots, v_n)$ of independent vectors the intersection ${\bf Crit}_{n,m}\cap (\{\underline v\}\times U_{n,m})$
 is homeomorphic to the corresponding intersection for the standard basis. That is so because by a linear transformation every independent $n$-tuple can be transformed to the standard basis.
We will compute the codimension of ${\bf Crit}_{n,m}$ which is the same as the codimension of ${\bf Crit}_{n,m}\cap (\{\underline \eps\}\times U_{n,m})$ in $\{\underline \eps\}\times U_{n,m}$. The condition that $\eps_i\in \partial D_P$ i.e. $P(\eps_i)=1$ reduces to the statement that the coefficient of the monomial $x_i^{2m}$ is equal to 1. The condition that we have is a critical point is $\frac\partial{\partial x_j}P(\eps_i)=0$ for $j\ne i$. It is equivalent to vanishing of the coefficient of the monomial $x_jx_i^{2m-1}$.
Therefore the codimension is equal to $n(n-1)$.
As shown by the example the sets of  ${\bf Degen}_{n,m}\cap(\{\underline v\}\times U_{n,m})$ are smaller, they are of codimension $n(n-1)+1$. Therefore $\dim({\bf Degen}_{n,m})=\dim V_{n,m}+\dim W(n)-(n(n-1)+1)<\dim(V_{n,m})$ and the projection ${\bf Degen}_{n,m}\to V_{n,m}$ cannot be surjective. It follows that for each $m>1$ the set
$$\{P\in U_{n,m}\;|\;\tilde g_P\text{ is Morse}\}$$
is open and dense in $U_{n,m}$.
The argument can be modified to obtain the same conclusion for nonhomogeneous polynomials.
\smallskip

The reasoning above is based on the codimension argument. For complex manifolds the codimension is always equal to the number of locally independent equations at a generic point. For real analytic sets the codimension may be bigger, but for the argument it is even better.

\section{Auerbach simplices}

Let $D\subset \R$ be a convex body. We call a simplex inscribed in $D$ (i.e. such that all its vertices lie on the boundary of $D$)  an {\it Auerbach simplex} of $D$ provided  every  vertex admits a supporting hyperplane parallel to the opposite face. The methods developed in the previous sections could be used to estimate the number of different Auerbach simplices of $D$.

Similarly to the case of Auerbach bases, any bifurcation point of the volume function
defined on the product of $n+1$  copies of $\partial D$. Below we briefly describe the homotopy type of the manifold of all such simplices.

\def\Simp{{\bf Simp}}
Let $\Simp$ be the space of all nondegenerate simplices in $\R^n$, i.e. $(n+1)$-tuples of points in $\R^n$ which do not lie on a common affine hyperplane.
Let $\Simp( D)\subset\Simp$, the space of nondegenerate simplices inscribed in $D$, i.e. $b_i\in\partial D$ for $i=0,1,\dots,n$.

\begin{proposition} The space $\Simp( D)$ is homotopy equivalent to $O(n)$\end{proposition}
\begin{proof} Let us consider the following topological spaces
\begin{enumerate}\item
$GL_n(\R)\times int(\Delta^n_{st})\times int(D),$ where
$$\Delta^n_{st}=\{(x_1,x_2,\dots,x_n)\in \R^n_{\geq0}\;:\;\sum_{i=1}^n x_i\leq 1\}\,.$$
is the standard simplex,

\item $$\Simp^{(2)}=\{(\Delta,a)\in \Simp\times int(D)\;:\;a\in int(\Delta)\}\,$$
the set of simplices with a chosen point of $int(D)$ in its interior,

\item $$\Simp^{(3)}
=\{(\Delta,a)\in \Simp(D)\times int(D)\;:\;a\in int(\Delta)\}\,$$
the set of inscribed simplices with a chosen point in its interior.
\end{enumerate}
We  show that the space (1) is homeomorphic to (2).
The homeomorphism (1)$\to$(2)
is the following: let $(v_1,v_2,\dots,v_n)\in GL_n(\R)$ (a matrix consisting of columns $v_i$), $(x_1,x_2,\dots,x_n)\in int(\Delta^n_{st})$, $a\in int(D)$. W define the homeomorphism by the formula:
$$\big((v_1,v_2,\dots,v_n),(x_1,x_2,\dots,x_n),a\big)
\mapsto \big(T_w(\langle a,a+v_1,a+v_2,\dots,a+v_n\rangle),a\big)\,,$$
where $T_w$ is the translation by the vector $w=-\sum x_iv_i$.
The inverse map is the following: given a simplex
 $\Delta=\langle b_0,b_1,\dots b_n\rangle$ and $a\in int(\Delta)\cap int(D)$. The element of $GL_n(\R)$ is defined  by
$$(v_1,v_2,\dots,v_n)=(b_1-b_0,b_2-b_0,\dots,b_n-b_0)\,.$$
The point in the standard simplex $(x_1,x_2,\dots,x_n)$ consists of coordinates of $a-b_0$ in the basis $v_1,v_2,\dots,v_n$.
\medskip

We will show that the  space (3)  is homotopy equivalent to (2).
For $a\in int(D)$ let
$$p_a:\R^n\setminus \{a\}\to \partial D$$ be the central projection from $a$ onto $\partial D$.
Consider the following map
$$pr:\Simp^{(2)}\to \Simp^{(3)}\,,$$
$$pr(\langle b_0,b_1,\dots b_n\rangle,a)=(\langle p_a(b_0),p_a(b_1),\dots p_a(b_n)\rangle,a)\,.$$
The fibers of this map are homeomorphic to $\R^{n+1}_+$. In fact $$\Simp^{(2)}\stackrel{\text{homeo}}\simeq \Simp^{(3)}\times \R^{n+1}_+\,.$$ Therefore it is a homotopy equivalence.
\medskip

Clearly the space $\Simp^{(3)}$ is  homotopy equivalent to $\Simp(D)$: it fibers over $\Simp(D)$ with the fibers homeomorphic to $int(\Delta^n_{st})$.
It follows that $$\Simp(D)\sim GL_n(\R^n)\sim O(n)\,.$$
\end{proof}

The connected component $\Simp(D)^+$  consisting of positively oriented simplices has homotopy type of $SO(n)$.
 The group of permutations $\Sigma_{n+1}$ acts freely on the spaces of simplices. The category of the quotient space $\Simp(D)/\Sigma_{n+1}=\Simp(D)^+/A_{n+1}$ is at least as the category of $SO(n)$ by Theorem \ref{26}.
The cohomology algebra $H^*(SO(n);\Z_2)$ is described e.g. in \cite[Th. 3D2]{Hat}. Its cup-length is $O(n\log n)$.
Precisely, $$\ell_{\Z_2}(SO(n))=\theta(n):=\sum_{i\geq0,\;2i+1<n} (p_{n,i}-1)\,,\quad \text{where }\quad p_{n,i}=\min\big\{2^k\,|\,k\in \Z\,,\;2^k(2i+1)\geq n\big\}$$
Hence
$$\cat(\Simp(D)/\Sigma_{n+1})\geq\theta(n)\,.$$
\medskip

If the body $D$ is general enough, then (as before) to give a lower bound of the number of critical points we use Morse theory. The rational cohomology behaves better when we pass to the quotient space $SO(n)/A_{n+1}$. The dimension does not change since the action of $A_{n+1}$ is trivial on cohomology. (The group $A_{n+1}$ can be identified with the group of orientation preserving isometries of a regular simplex in $\R^n$.) We obtain
$$\text{number of critical points }\geq \dim H^*(SO(n)/A_{n+1};\Q)=\dim H^*(SO(n);\Q)=2^{\lfloor \frac n2\rfloor}\,.$$
The computation of the rational cohomology of $SO(n)$ follows from \cite[Prop. 3D4]{Hat}.

\section{Summary}
Below we list bounds for a number of bifurcation points for the spaces considered above.
For Auerbach bases:
\vskip5mm
\def\wys{\phantom{\begin{matrix}A^A\\A^A\end{matrix}}}
\begin{center}\begin{tabular}{|c|c|c|}
  \hline
 $\wys$ & Morse function  & arbitrary function\\
  \hline
  real & $\geq 2^{\lfloor\frac n2\rfloor}+4\wys$ & $\geq \frac{n(n-1)}2+1\wys$\\
  \hline
  complex & $\geq n!+2\wys$ & $\geq \frac{n(n-1)}2+1\wys$ \\
  \hline
\end{tabular}
\end{center}
\vskip10mm
For Auerbach simplices:
\vskip5mm
\def\wys{\phantom{\begin{matrix}A^A\\A^A\end{matrix}}}
\begin{center}\begin{tabular}{|c|c|c|}
  \hline
 $\wys$ & Morse function  & arbitrary function\\
  \hline
  real & $\geq 2^{\lfloor\frac n2\rfloor}+4\wys$ & $\gtrsim n\log n\wys$\\
  \hline
  \end{tabular}
\end{center}

\vskip15mm

\section{Appendix}

For the reader's convenience we repeat the definition of a bifurcation point, the Lusternik--Schnirelmann category and the theorem which we prove.

\begin{definition}\label{bifur2} We say that $x\in M$ is a topologically regular point of $f$ if there exists a neighbourhood of $x$ which is of a product form $U\simeq S\times (a-\eps,a+\eps) $ and the function $f$ on $U$ coincides with the projection onto the second factor. If $x$ is not topologically regular, then we say that it is a bifurcation point. The corresponding value $f(x)$ is called a bifurcation value.\end{definition}

\begin{remark}\rm In this definition we do not assume that $M$ is a manifold.\end{remark}

There are several variants of category of a space. The following is the most convenient for us:

\begin{definition} Let $X$ be a topological space and $Y\subset X$ its closed subspace.
Then the Lusternik--Schnirelmann category $\cat_X(Y)$ denots the smallest cardinality of
covering of $Y$ by open sets which are contractible in $X$. If $X=Y$ we write $\cat(X)=\cat_X(X)$.
\end{definition}

\begin{theorem}\label{LSmain2} Let $M$ be a path connected metric space which is locally contractible. Let $f:M\to \R$ be a continuous proper function which is bounded from below. For each $a\in \R$ the number of bifurcation points with $f(x)\leq a$ is not smaller than $\cat_M(M_{\leq a})$.\end{theorem}

\begin{remark}\rm It is enough to assume that each point has a neighbourhood which is contractible in the whole $M$.\end{remark}

The proof is based on the following construction.  For $i>0$ let
\begin{equation}\label{minmax}\lambda_i=\inf\big\{\sup\{f(x)\;|\;x\in X\}\;|\;X\subset M,\; \cat_M(X)\geq i\big\}\,,\end{equation}
where the infimum is taken over the compact subsets of $M$. Note that the above infimum is always attained for some set $X\subset M$. It is so because the space of all closed subsets of a compact metric space endowed with the Hausdorff distance is compact (cf. \cite[VI.28]{Hau}, \cite{Pri}) and the function
$X\mapsto \cat_M(X)$ is upper semi-continuous on it (any open covering of the limiting set covers sufficiently close sets too).

Clearly  if \begin{equation}\label{szacowaniecat}X\subset M_{<\lambda_i}\quad \text{then} \quad \cat_M(X)<i\,.\end{equation}

\begin{remark}\rm If $\lambda_{i+1}>\lambda_i$ then the infimum of (\ref{minmax}) is realized by $M_{\leq\lambda_i}$. That is so because if $X$ realizes infimum, then $X\subset M_{\leq\lambda_i}$, hence $\cat_M(M_{\leq\lambda_i})\geq \cat_M(X)=i$.  On the other hand $\cat_M(M_{\leq\lambda_i})<i+1$ by (\ref{szacowaniecat}).\end{remark}

\begin{lemma}\label{lem1} Each $\lambda_i$ is a bifurcation value.\end{lemma}

To show that we need to construct an isotopy which is an analogue of a gradient flow.

\begin{proof} Let $x\in M$ be a topologically regular point with value $f(x)=a$.
We claim that there exist a homeomorphism $\phi:M\to M$, such that
$f(\phi(y))\leq f(y)$ for all $y\in M$,  and $f(\phi(x))<f(x)$.
According to the Definition \ref{bifur2} there exists a neighbourhood of product form $U\simeq S\times (a-\eps,a+\eps)$ of $x=(s_0,a)$.
The homeomorphism  $\phi$ will be the identity outside $U$.
Inside  $U$ we define it as follows. Let  $\alpha: S\to \R$ be a nonnegative continuous function, bounded by $\tfrac \eps 2$ with compact support. Similarly, let $\beta:(a-\eps,a+\eps)\to \R$ be a nonnegative smooth function, bounded by $\tfrac \eps 2$, $|\alpha'|<1$, with support contained in $[a-\tfrac\eps 2,a+\tfrac\eps 2]$.
We assume $\alpha(s_0)>0$ and $\beta(a)>0$.
Then with the identification $U\simeq S\times (a-\eps,a+\eps)$ we put
$$\phi(x)=\phi(s,t)=(s,t-\alpha(s)\beta(t))\,.$$

Suppose that $a$ is not a bifurcation value. For each
point $x\in f^{-1}(a)$ we chose a homeomorphism $\phi_{x}$ described above. Let $$V_x=\{y\in f^{-1}(a)\;:\;f(\phi_{x}(y))<f(y)\}\,.$$
This is an open cover of $f^{-1}(a)$. The set $f^{-1}(a)$ is compact, so we can choose a finite subcover indexed by $x_1,x_2,\dots,x_m$. The homeomorphism $\Phi=\phi_{x_1}\circ\phi_{x_2}\circ\dots \phi_{x_m}$ has the property, that
$$f(\Phi(x))<a\quad \text{for}\quad x\in M_{\leq(a)}\,.$$

 Suppose $X\subset M$ is the subset realizing the infimum of (\ref{minmax}) with $\sup_X(f(x))=\lambda_i$.  Then if $\Phi$ is the above homeomorphism constructed for $a=\lambda_i$, $\Phi(X)$ has the same category in $M$. On the other hand $\sup_{\Phi(X)}(f(x))<\lambda_i$.
This contradicts minimality of $X$.
\end{proof}

\begin{lemma} \label{lem2} Suppose the bifurcation set is discrete. Then $$\lambda_i<\lambda_{i+1}$$ for $i=1,2, \dots, \cat(M)-1$.\end{lemma}

\begin{proof} Suppose $\lambda_i=\lambda_{i+1}$ and $X\subset M_{\leq \lambda_i}$ realizes the value $\lambda_{i+1}$ in (\ref{minmax}). Let  $x_1,x_2,\dots,x_m$ be the bifurcation points lying on $f^{-1}(\lambda_i)$. Let $V_1,V_2,\dots,V_m$ be contractible neighbourhoods of $x_1,x_2,\dots,x_m$. We can assume that these sets are disjoint. Let $V=\bigcup_{k=1}^mV_k$.  By the method of the previous Lemma we construct a homeomorphism $\Phi$ of $M$ such that $f(\Phi(x))<f(x)$ for $x\in f^{-1}(\lambda_i)\setminus V$. We have
$$\cat_M(X\setminus V)=\cat_M(\Phi(X\setminus V))<i$$
by (\ref{szacowaniecat}).
This means that $X\setminus V$ can be covered by $i-1$ contractible sets. The set $V$ is contractible in $M$ since we assume that $M$ is path connected. Hence $\cat_M(X)\leq i$ which is contradiction.
\end{proof}

\noindent{\it Proof of Theorem \ref{LSmain2}.}
If the bifurcation set is not discrete, then it is infinite and we are done. Suppose it is discrete. For $b$ sufficiently negative the set $M_{\leq b}$ is empty, thus it has category equal to zero. Each value $\lambda_i$ for $i=1,2,\dots \cat(M_{\leq a})$ appears in the segment $[b,a]$. By Lemma \ref{lem1} for each $i$ there is a bifurcation point at the level set $f^{-1}(\lambda_i)$. By Lemma \ref{lem2} there are $\cat_M(M_{\leq a})$ distinct values of $\lambda_i$.\qed

\begin{remark} \rm Without the assumption that the bifurcation set is discrete one can show as \cite[19.12]{DFN} that if $\lambda_i=\lambda_{i+p}$, then the category of the bifurcation set is greater than $p$.\end{remark}.

\end{document}